\documentclass[11pt]{amsart}

\usepackage[margin=1in]{geometry}
\usepackage{amsmath}
\usepackage{amssymb}
\usepackage{amsthm}
\usepackage{cite}
\usepackage{microtype}
\usepackage[hidelinks]{hyperref}

\usepackage{comment}

\usepackage{etoolbox}
\apptocmd{\thebibliography}{\raggedright}{}{}

\newcommand{\R}{\mathbf{R}}
\newcommand{\Z}{\mathbf{Z}}
\newcommand{\N}{\mathbf{N}}
\newcommand{\F}{\mathbf{F}}
\newcommand{\E}{\mathbf{E}}

\newcommand{\A}{\mathbf{A}}
\newcommand{\C}{\mathbf{C}}

\newcommand{\norm}[1]{\| #1 \|}
\newcommand{\abs}[1]{\lvert #1 \rvert}

\renewcommand{\Re}{\text{Re}}
\renewcommand{\mod}[1]{\;(\operatorname{mod}\;#1)}

\newtheorem{thm}{Theorem}

\newtheorem{cor}{Corollary}

\newtheorem{lem}{Lemma}

\newtheorem{prop}{Proposition}

\theoremstyle{remark}
\newtheorem*{rem}{Remark}

\theoremstyle{definition}
\newtheorem*{defn}{Definition}

\title{Small gaps between configurations of prime polynomials}
\author{Hans Parshall}

\begin{document}

\begin{abstract}
	We find arbitrarily large configurations of irreducible polynomials over finite fields that are separated by low degree polynomials.  Our proof adapts an argument of Pintz from the integers, in which he combines the methods of Goldston-Pintz-Y\i ld\i r\i m and Green-Tao to find arbitrarily long arithmetic progressions of generalized twin primes.
\end{abstract}

\maketitle

\section{Introduction}

A beautiful theorem of Pintz \cite[Theorem 1]{pintz} asserts that, in the integers, there are arbitrarily long arithmetic progressions of generalized twin primes (that is, primes which are separated from another prime by some absolute constant).  This built on his work with Goldston and Y\i ld\i r\i m \cite{gpy}, in which they showed that, conditional on a version of the Elliot-Halberstam conjecture, there are infinitely many generalized twin primes.  Remarkably, Zhang \cite{zhang} was able to expand on their work to prove, unconditionally, the existence of infinitely many generalized twin primes.  Zhang's techniques were sufficient for Pintz to establish his theorem unconditionally as well (see \cite{pintz-polignac}).

Independently of Zhang (and one another), Maynard \cite{maynard} and Tao \cite{taoblog} were able to use more elementary arguments to significantly shrink the best known bound on small gaps between primes.  They were further able to establish bounded gaps between infinitely many sets of $m + 1$ consecutive primes for all $m \in \N$.  The combination of the methods of Zhang, Maynard, and Tao led to the state of the art in this direction by the Polymath project \cite{polymath8b}.

Each of these results takes the common approach of proving an approximation to the Hardy-Littlewood $k$-tuples conjecture.  We call a finite set $\mathcal{H} \subseteq \Z$ admissible if for every prime $p$, we have $\mathcal{H} + (p) \neq \Z$.  Then the conjecture of Hardy and Littlewood states that, for any admissible $\mathcal{H}$, there exist infinitely many integers $n$ where the set $n + \mathcal{H} = \{n + h : h \in \mathcal{H}\}$ contains only primes.  Pintz proved his theorem by combining this approach with the proof by Green and Tao \cite{greentao} that the primes contain arbitrarily long arithmetic progressions.  By combining the work of Polymath \cite[Remark 32]{polymath8b} with one of Pintz' main results \cite[Theorem 5]{pintz}, one can deduce that many translates of any admissible $\mathcal{H}$ produce arbitrarily long arithmetic progressions of primes.

\begin{thm}\label{headline-int}
Let $m \in \N$, and let $k \in \N$ be taken sufficiently large with respect to $m$.  Then for any $\ell \in \N$ and any admissible $\mathcal{H} = \{h_1, \ldots, h_k\} \subseteq \Z$, there are infinitely many $n, d \in \N$ where at least $m + 1$ of the arithmetic progressions $h_j + \{n + id\}_{i = 0}^{\ell - 1}$ contain only primes.
\end{thm}

With $m = 1$, Polymath was able to take $k = 50$ with an explicit admissible $\mathcal{H}$ to show the existence of infinitely many pairs of primes at most 246 apart \cite[Theorem 4]{polymath8b}.  The same $\mathcal{H}$ yields a quantitative improvement for long arithmetic progressions of generalized twin primes.

\begin{cor}\label{headline-int-quantitative}
For any $\ell \in \N$, there exists $h \in \N$ with $h \leq 246$ such that there exist infinitely many $n, d \in \N$ where both arithmetic progressions $\{n + id\}_{i = 0}^{\ell - 1}$ and $h + \{n + id\}_{i = 0}^{\ell - 1}$ contain only primes.
\end{cor}

In this paper, we are interested in proving comparable results for polynomials over a finite field. From now on, we will work in $\F_q[t]$ with $q$ a fixed prime power.  The notion of an $\ell$-term arithmetic progression will be replaced by an $\ell$-configuration, defined for $f,g \in \F_q[t]$ by 
\[
\mathcal{C}_\ell(f,g) = \{f + gh : \deg(h) < \ell\},
\]
and by a prime in $\F_q[t]$ we mean a monic irreducible polynomial.  The analogue of the Green-Tao theorem in $\F_q[t]$ was established by L\^e \cite{le}, in which he proved that there are arbitrarily large $\ell$-configurations containing only primes. As in the integers, we call a finite set $\mathcal{H} \subseteq \F_q[t]$ admissible provided that, for every prime $p \in \F_q[t]$, we have $\mathcal{H} + (p) \neq \F_q[t]$.  Our main result is the following.
\begin{thm}\label{headline}
Let $m \in \N$, and let $k \in \N$ be taken sufficiently large with respect to $m$. Then for any $\ell \in \N$ and any admissible $\mathcal{H} = \{h_1, \ldots, h_k\} \subseteq \F_q[t]$, there are infinitely many $f,g \in \F_q[t]$ with $g \neq 0$ such that at least $m + 1$ of the $\ell$-configurations $h_j + \mathcal{C}_\ell(f,g)$ contain only primes.
\end{thm}

We do not attempt to reproduce the strongest quantitative results of Polymath, which would require adapting the impressive work of Zhang to $\F_q[t]$. Instead, we settle for the relationship between $m$ and $k$ established in \cite[Theorem 23]{polymath8b}, taking $k \geq 54$ when $m = 1$. Then as an immediate corollary to \autoref{headline}, we can obtain $\ell$-configurations of ``twin primes'' by noting that for any degree $d \geq 0$, the set $\{\alpha t^d : \alpha \in \F_q^\times\}$ is admissible.

\begin{cor}\label{headline-quant}
Let $\ell \in \N$. If $q \geq 55$, then for any degree $d \geq 0$, there exists a monomial $h \in \F_q[t]$ with degree $d$ such that there exist infinitely many $f,g \in \F_q[t]$ with $g \neq 0$ where both $\ell$-configurations $\mathcal{C}_\ell(f,g)$ and $h + \mathcal{C}_\ell(f,g)$ contain only primes.
\end{cor}

The argument of Pintz in \cite{pintz} splits nicely into two steps.  First, he uses the Goldston-Pintz-Y\i ld\i r\i m method to show that the set of almost-prime translates of an admissible $\mathcal{H} \subseteq \Z$ with at least two primes form a positive density subset of the almost-prime translates of $\mathcal{H}$.  In \autoref{densitysection}, we will mimic this step by following the recent work on bounded gaps between primes in the integers, especially following Tao \cite{taoblog}.  Second, Pintz generalizes the argument of Green and Tao from subsets of almost-primes to subsets of almost-prime translates of $\mathcal{H}$.  In \autoref{transsection}, we will similarly generalize the argument of L\^e.  We combine these ingredients in \autoref{proofmain} to give the proof of \autoref{headline}.

\section{Setup}\label{setup}

It will be convenient to have a common setup throughout this paper.  Fix an integer $m \in \N$. From \cite[Theorem 23]{polymath8b}, for all $k \in \N$ taken sufficiently large with respect to $m$ (with $k \geq 54$ when $m = 1$), we can fix a smooth $F : [0,\infty)^k \rightarrow \R$, supported on the $k$-simplex $\{\mathbf{t} \in [0,1]^k : t_1 + \cdots + t_k \leq 1\}$, with $F$ symmetric in the variables $t_1, \ldots, t_k$ and $k \beta_F / \alpha_F > 4m$, where $\alpha_F$ and $\beta_F$ are the positive real numbers depending only on $F$ defined in \autoref{asymptotic}. Fixing such a $k$ and $F$, we also fix $\eta \in (0,1/2)$ sufficiently close to 1/2 so that $k\beta_F / \alpha_F > 2m / \eta$.  This crucial inequality relating $k, F, m$ and $\eta$ will be important in the proof of \autoref{positivedensity}.  Finally, we fix an admissible $\mathcal{H} = \{h_1, \ldots, h_k\} \subseteq \F_q[t]$ with $\deg(h_1) \leq \cdots \leq \deg(h_k)$.  Constants implied by big-$O$ notation and its shorthand $\ll$ will often depend on these fixed quantities.

The setting for the arguments to follow will be 
\[
\A_n = \{f \in \F_q[t] : f \text{ monic with } \deg(f) = n\}
\]
for large degrees $n \in \N$.  When considering divisors of polynomials in $\A_n$, we will often restrict our attention to polynomials of degree up to $r = \eta n$.  We will denote the set of primes in $\F_q[t]$ by $\mathcal{P}$, and to avoid complications from small primes, we will take a sufficiently large degree $w \in \N$ and work modulo 
\[
	W = \prod_{\substack{p \in \mathcal{P} \\ \deg(p) < w}} p.
\]  
Although we will eventually need to fix $w$ in the proof of \autoref{headline}, it is helpful to think of $w$ tending to infinity up until that point, and we use the asymptotic notation $o(1)$ for a quantity tending to zero as $w \rightarrow \infty$.  Since $\mathcal{H}$ is admissible, we can take $w > \deg(h_k)$ and fix a congruence class $b \mod{W}$ with $(W, b + h_j) = 1$ for each $h_j \in \mathcal{H}$. We allow $n \rightarrow \infty$, and we always insist that $w \ll \log \log n$ so that $\deg(W) \ll \log n$ and terms that tend to zero as $n \rightarrow \infty$ are also $o(1)$.

For a friendly introduction to the subject of number theory in function fields, see Rosen's excellent text \cite{rosen}.  We will need some notions in $\F_q[t]$ that are similar to their integer counterparts, such as the prime number theorem \cite[Theorem 2.2]{rosen} and Dirichlet's theorem \cite[Theorem 4.8]{rosen}. We will use some familiar arithmetic functions on $\F_q[t]$, namely the norm $\abs{f} = q^{\deg(f)}$, the Euler totient function $\varphi(f)$, the M\"obius function $\mu(f)$, and a weighted characteristic function of $\mathcal{P}$, $\theta(f) = \deg(f) 1_{\mathcal{P}}(f)$.  In our asymptotics, a large role is played by the zeta function
\[
	\zeta(s) = \sum_{\substack{f \in \F_q[t] \\ f \text{ monic}}} \frac{1}{|f|^s},
\]
which (like the Riemann zeta function) can be continued to a meromorphic function on $\C$ with a simple pole at $s = 1$ with residue $1/\log(q)$.  For $\Re(s) > 1$, we will use the Euler product expansion
\[
	\zeta(s) = \prod_{p \in \mathcal{P}} \left(1 - \frac{1}{\abs{p}^s}\right)^{-1}.
\]

\section{The Density Argument}\label{densitysection}

For $\epsilon > 0$, we define the almost-prime translates of $\mathcal{H}$ by
\[
	\mathcal{P}_\epsilon(\mathcal{H}) = \left\{f \in \F_q[t] : P^-\left(\prod_{j = 1}^k (f + h_j)\right) \geq \epsilon n\right\},
\]
where $P^-(f)$ denotes the least degree of the prime factors of $f \in \F_q[t]$.  Our goal in this section is to show that many $f \in \mathcal{P}_\epsilon(\mathcal{H})$ result in $m + 1$ primes among the $f + h_j$. More precisely, we will prove the following theorem.

\begin{thm}\label{positivedensity}
	Fix $w \in \N$ sufficiently large and $\epsilon \in (0,1)$ sufficiently small.  Then there exists some $a \in (0,1)$ such that, for all $n \in \N$ with $\log\log(n) \geq w$, the set 
	\[
		\mathcal{A} = \left\{f \in \mathcal{P}_\epsilon(\mathcal{H}) \cap \A_n : f \equiv b \mod{W} \text{ and at least } m + 1 \text{ of the } f + h_j \text{ are prime}\right\}
	\]
	satisfies
	\begin{equation}\label{positivedensitybound}
		\abs{\mathcal{A}} \geq a \frac{\abs{W}^{k - 1}}{\varphi(W)^k} \frac{\abs{\A_n}}{n^k} 
	\end{equation}
\end{thm}

\begin{rem}
	One can see from the proof that $w$ and $\epsilon$ depend only on our fixed quantities from \autoref{setup}.
\end{rem}

Except for the use of the crucial inequality in the proof of \autoref{positivedensity}, we do not use anything special about the Polymath function $F$.  In particular, our main asymptotics, contained in \autoref{npsums}, and our concentration estimate, contained in \autoref{concentration}, are valid for any smooth, symmetric function supported on the $k$-simplex.  One could even drop the symmetry condition by redefining our constants depending on $F$ a bit more carefully (see \cite[Lemma 41]{polymath8b}).

In what follows, the variables $d, d', d_1, \ldots, d_k, d_1', \ldots, d_k'$ will always denote monic polynomials in $\F_q[t]$ with $[d, d'], [d_1, d_1'], \ldots, [d_k,d_k'], W$ all pairwise coprime.  It will be convenient to introduce shorthand for our weights, supported on such $d$, by
\[
	\lambda_{d_1, \ldots, d_k} = \left(\prod_{j = 1}^k \mu(d_j)\right) F\left(\frac{\deg(d_1)}{r}, \ldots, \frac{\deg(d_k)}{r}\right).
\]
Similar weights were used in the integer setting in \cite{taoblog}, where Tao establishes the integer analogue of \autoref{npsums}.  These differ from the integer weights of Maynard in \cite{maynard}, but produce essentially the same results.  A version of \autoref{npsums} for Maynard's weights in $\F_q[t]$ appeared in \cite{chlopt}, but it is unclear how to adapt their setup to obtain a concentration estimate comparable to \autoref{concentration}.  Such an estimate was the heart of Pintz' argument in \cite{pintz} (see also \cite[Proposition 14]{polymath8b}).

We now state our main asymptotics and concentration estimate, and derive \autoref{positivedensity}.

\begin{prop}\label{npsums} Let $h_j \in \mathcal{H}$ and set
\begin{align*}
	S_1 &= \sum_{\substack{f \in \A_n \\ f \equiv b \mod{W}}} \left(\sum_{\substack{d_1, \ldots, d_k \\ d_j \mid f + h_j \forall j}} \lambda_{d_1, \ldots, d_k} \right)^2,\\
	S_2 &= \sum_{\substack{f \in \A_n \\ f \equiv b \mod{W}}} \theta(f + h_j) \left(\sum_{\substack{d_1, \ldots, d_k  \\ d_j \mid f + h_j \forall j}} \lambda_{d_1, \ldots, d_k} \right)^2. 
\end{align*}
Then with $\alpha_F$ and $\beta_F$ as in \autoref{asymptotic}, we have the asymptotics
\[
	S_1 = (\alpha_F + o(1)) \frac{\abs{W}^{k - 1}}{\varphi(W)^k} \frac{\abs{\A_n}}{r^k}  \hspace{44pt} S_2 = (\beta_F + o(1)) \frac{\abs{W}^{k - 1}}{\varphi(W)^k} \frac{\abs{\A_n}}{r^{k - 1}} 
\]
\end{prop}

\begin{rem}
	One can use \autoref{npsums}, together with \eqref{positivedensitybound}, to show that $f \in \mathcal{P}_\epsilon(\mathcal{H})$ with at least $m + 1$ of the $f + h_j$ prime form a set of positive upper density in $\mathcal{P}_\epsilon(\mathcal{H})$.
\end{rem}

\begin{prop}\label{concentration} For any $\epsilon \in (0,\eta)$,
\[
	\sum_{\substack{f \in \A_n \setminus \mathcal{P}_\epsilon(\mathcal{H}) \\ f \equiv b \mod{W}}} \left(\sum_{\substack{d_1, \ldots, d_k \\ d_j \mid f + h_j \forall j}} \lambda_{d_1, \ldots, d_k} \right)^2 \ll \epsilon \frac{\abs{W}^{k - 1}}{\varphi(W)^k} \frac{\abs{\A_n}}{r^k} 
\]
\end{prop}

\begin{proof}[Proof of \autoref{positivedensity}] We will estimate
\begin{equation}\label{maincount}
	\sum_{\substack{f \in \A_n \cap \mathcal{P}_\epsilon(\mathcal{H}) \\ f \equiv b \mod{W}}} \left(\sum_{j = 1}^k \theta(f + h_j) - mn\right) \left(\sum_{\substack{d_1, \ldots, d_k \\ d_j \mid f + h_j \forall j}} \lambda_{d_1, \ldots, d_k}\right)^2,
\end{equation}
since the summands are positive exactly for $f \in \mathcal{P}_\epsilon(\mathcal{H})$ with at least $m + 1$ of the $f + h_j$ prime.  A lower bound for \eqref{maincount} is given by
\begin{align*}
	\sum_{\substack{f \in \A_n \\ f \equiv b \mod{W}}}&\left(\sum_{j = 1}^k \theta(f + h_j) - mn\right) \left(\sum_{\substack{d_1, \ldots, d_k \\ d_j \mid f + h_j \forall j}} \lambda_{d_1, \ldots, d_k}\right)^2\\
	&- \sum_{\substack{f \in \A_n \setminus \mathcal{P}_\epsilon(\mathcal{H}) \\ f \equiv b \mod{W}}} \left(\sum_{j = 1}^k \theta(f + h_j)\right) \left(\sum_{\substack{d_1, \ldots, d_k \\ d_j \mid f + h_j \forall j}} \lambda_{d_1, \ldots, d_k}\right)^2,
\end{align*}
which by \autoref{npsums} and \autoref{concentration} is at least
\[
	\left(k(\beta_F - o(1))  - \frac{m(\alpha_F + o(1))}{\eta} - O(\epsilon)\right) \frac{\abs{W}^{k - 1}}{\varphi(W)^k} \frac{\abs{\A_n}}{r^{k - 1}} .
\]
Taking $w$ sufficiently large and $\epsilon$ sufficiently small, our choice of $F$ guarantees that \eqref{maincount} is at least
\[
	c\frac{\abs{W}^{k - 1}}{\varphi(W)^k} \frac{\abs{\A_n}}{n^{k - 1}} 
\]
for some $c > 0$.  A crude upper bound for \eqref{maincount} is given by
\[
	Cn\sum_{f \in \mathcal{A}} \left(\sum_{\substack{d_1, \ldots, d_k \\ d_j \mid f + h_j \forall j}} 1 \right)^2
\]
for some $C > 0$.  We need only consider $d_j$ built out of the prime divisors of $\prod_{j = 1}^k (f + h_j)$, of which there are at most $1/\epsilon$.  Then there are at most $2^{1/\epsilon}$ such choices for each $d_j$, so our upper bound for \eqref{maincount} is at most $Cn\abs{\mathcal{A}}$ for some new $C > 0$, establishing \eqref{positivedensitybound}.
\end{proof}

It remains to prove both \autoref{npsums} and \autoref{concentration}.  In \autoref{eulerprod} below, we give some reasonably standard Euler product estimates that will be common to our arguments.

\begin{lem}\label{eulerprod} For $x, x' \in \R$, set
\[
	K(x, x') = \prod_{\substack{p \in \mathcal{P} \\ p \nmid W}} \left(1 + \frac{K_p(x,x')}{\abs{p}}\right)
\]
where
\[
	K_p(x, x') = \sum_{\substack{d,d' \\ [d,d'] = p}} \frac{\mu(d)\mu(d')}{|d|^{\frac{1 + ix}{r \log q}}|d'|^{\frac{1 + ix'}{r \log q}}}
\]
Then $K(x,x') = O(r^3)$ uniformly in $x, x'$.  Further, for $\abs{x}, \abs{x'} \leq \sqrt{r}$, we have the asymptotic
\[
	K(x,x') = \left(1 + o(1)\right) \frac{\abs{W}}{\varphi(W) r} \frac{(1 + ix)(1 + ix')}{2 + ix + ix'}.
\]
\end{lem}

\begin{proof}[Proof of \autoref{eulerprod}] Note that we have
	\[
		K_p(x,x') = - \frac{1}{\abs{p}^{1 + \frac{1 + ix}{r \log q}}} - \frac{1}{\abs{p}^{1 + \frac{1 + ix'}{r \log q}}} + \frac{1}{\abs{p}^{1 + \frac{2 + ix + ix'}{r \log q}}}.
	\]
	Then we have the estimate
	\[
		1 + K_p(x,x') = \left(1 + O\left(\frac{1}{\abs{p}^2}\right)\right) \frac{\left(1 - \abs{p}^{-1 - \frac{1 + ix}{r \log q}}\right)\left(1 - \abs{p}^{-1 - \frac{1 + ix'}{r \log q}}\right)}{1 - \abs{p}^{-1 - \frac{2 + ix + ix'}{r \log q}}},
	\]
	which can be verified by expanding $\left(1 - \abs{p}^{-1 - \frac{2 + ix_j + ix_j'}{r \log q}}\right)^{-1}$ as a geometric series. Then from
	\[
		\prod_{\substack{p \in \mathcal{P} \\ p \nmid W}} \left(1 + O\left(\frac{1}{\abs{p}^2}\right)\right) = 1+ o(1),
	\]
	we have the estimate
	\begin{equation}\label{modifiedzeta}
		K(x,x') = (1 + o(1)) \frac{\zeta_W\left(1 + \frac{2 + ix + ix'}{r \log q}\right)}{\zeta_W\left(1 + \frac{1 + ix}{r \log q}\right)\zeta_W\left(1 + \frac{1 + ix'}{r \log q}\right)},
	\end{equation}
	where we define the modified zeta function $\zeta_W$ in the range $\Re(s) > 1$ by
	\[
		\zeta_W(s) = \prod_{\substack{p \in \mathcal{P} \\ p \nmid W}} \left(1 - \frac{1}{\abs{p}^s}\right)^{-1}.
	\]
	In the range $\Re(s) \geq 1 + \frac{1}{r\log q}$, we have both $\abs{\zeta_W(s)}, \abs{\zeta_W(s)^{-1}} \ll r$, establishing the estimate $K(x,x') = O(r^3)$.  To establish our claimed asymptotic, we need to compare $\zeta_W(s)$ to $\zeta(s)$ by accounting for the primes $p \mid W$. Restricting our attention to $\abs{x} \leq \sqrt{r}$, we have the estimate
\[
	1 - \frac{1}{\abs{p}^{1 + \frac{1 + ix}{r \log q}}} = 1 - \frac{1}{\abs{p}} + O\left(\frac{\log \abs{p}}{\abs{p}\sqrt{r}}\right),
\]
in which case
\[
	\prod_{\substack{p \in \mathcal{P} \\ p \mid W}} \left(1 - \frac{1}{\abs{p}^{1 + \frac{1 + ix}{r \log q}}}\right) = \frac{\varphi(W)}{\abs{W}} \exp\left(O\left(\sum_{p \mid W} \frac{\log\abs{p}}{\abs{p}\sqrt{r}}\right)\right) = (1 + o(1))\frac{\varphi(W)}{\abs{W}}.
\]
Using this to replace $\zeta_W$ by $\zeta$ in \eqref{modifiedzeta}, we can apply the asymptotic
\[
	\zeta(s) = \left(\frac{1}{\log q} + o(1)\right) \frac{1}{s - 1}
\]
from the pole at $s = 1$, certainly valid for $s = 1 + O(1/r)$, to obtain our claimed asymptotic.
\end{proof}

The following lemma contains the key asymptotics used in the proof of \autoref{npsums}.

\begin{lem}\label{asymptotic} Set
\begin{align*}
	S_3 &= \sum_{d_1, \ldots, d_k, d_1', \ldots, d_k'} \frac{\lambda_{d_1, \ldots, d_k} \lambda_{d_1', \ldots, d_k'}}{\prod_{j = 1}^k |[d_j, d_j']|}\\
	S_4 &= \sum_{d_2, \ldots, d_k, d_2', \ldots, d_k'} \frac{\lambda_{1, d_2, \ldots, d_k} \lambda_{1, d_2', \ldots, d_k'}}{\prod_{j = 2}^k \varphi([d_j, d_j'])}
\end{align*}
Then we have the asymptotics
\[
	S_3 = (\alpha_F + o(1)) \left(\frac{\abs{W}}{\varphi(W) r}\right)^k  \hspace{44pt} S_4 = (\beta_F + o(1)) \left(\frac{\abs{W}}{\varphi(W) r}\right)^{k - 1} ,
\]
where
\[
	\alpha_F = \int_{[0,\infty)^k} \frac{\partial F}{\partial t_1 \cdots \partial t_k} (\mathbf{t})^2 \;d\mathbf{t} \hspace{44pt} \beta_F = \int_{[0,\infty)^{k - 1}} \frac{\partial F}{\partial t_2 \cdots \partial t_k} (0, t_2, \ldots, t_k)^2\;dt_2 \cdots dt_k.
\]
\end{lem}

\begin{proof}[Proof of \autoref{asymptotic}]
	We begin by extending the function $\mathbf{t} \mapsto \left(\prod_{j = 1}^k e^{t_j}\right) F(\mathbf{t})$ to a smooth, compactly supported function on all of $\R^k$, so that we have the Fourier expansion
	\begin{equation}\label{ffourier}
		\left(\prod_{j = 1}^k e^{t_j}\right) F(\mathbf{t}) = \int_{\R^k} \left(\prod_{j = 1}^k e^{-it_jx_j}\right) \tilde{F}(\mathbf{x})\;d\mathbf{x},
	\end{equation}
	where $\tilde{F} : \R^k \rightarrow \C$ is rapidly decreasing in the sense that $\tilde{F}(\mathbf{x}) = O\left((1 + \abs{\mathbf{x}})^{-A}\right)$ for any fixed $A > 0$ and all $\mathbf{x} \in \R^k$.  Then from $\deg(d_j)/r = \log |d_j| / (r\log q)$ and \eqref{ffourier} we can write 
	\begin{equation}\label{frfourier}
		F\left(\frac{\deg(d_1)}{r}, \ldots, \frac{\deg(d_k)}{r}\right) = \int_{\R^k} \left(\prod_{j = 1}^k \abs{d_j}^{-\frac{1 + ix_j}{r \log q}}\right) \tilde{F}(\mathbf{x})\;d\mathbf{x}.
	\end{equation}
	Note that, uniformly in $x, x'$, we have the estimate
	\begin{align*}
		\left| \sum_{d, d' } \frac{\mu(d)\mu(d')}{|[d,d']| |d|^{\frac{1 + ix}{r \log q}}| |d'|^{\frac{1 + ix'}{r \log q}}} \right|	&\leq \prod_{p \in \mathcal{P}} \left(1 + \frac{2}{|p|^{1 + \frac{1}{r\log q}}} + \frac{1}{|p|^{1 + \frac{2}{r\log q}}}\right)\\
		&\leq \exp\left(O(\log r)\right).
	\end{align*}
	Then we are justified in using \eqref{frfourier} and Fubini to write
	\[
		S_3 = \int_{\R^k}\int_{\R^k} \left(\prod_{j = 1}^k \sum_{d, d'} \frac{\mu(d)\mu(d')}{|[d,d']| |d|^{\frac{1 + ix_j}{r \log q}}| |d'|^{\frac{1 + ix_j'}{r \log q}}} \right) \tilde{F}(\mathbf{x})\tilde{F}(\mathbf{x}')\;d\mathbf{x}\;d\mathbf{x'}.
	\]
	Recall that we're restricting our attention to $d, d'$ with $W$ and $[d,d']$ coprime.  Then we can factor the above sums in $d,d'$ as Euler products to obtain
	\[
		S_3 = \int_{\R^k} \int_{\R^k} \left(\prod_{j = 1}^k K(x_j, x_j')\right) \tilde{F}(\mathbf{x})\tilde{F}(\mathbf{x'})\;d\mathbf{x}\;d\mathbf{x'},
	\]
	where $K$ is as in \autoref{eulerprod}. Using our uniform bound on $K$ and the rapid decrease of $\tilde{F}$, we have
	\[
		S_3 = (1 + o(1))\int_{|\mathbf{x}| \leq \sqrt{r}} \int_{|\mathbf{x'}| \leq \sqrt{r}} \left(\prod_{j = 1}^k K(x_j, x_j')\right) \tilde{F}(\mathbf{x})\tilde{F}(\mathbf{x'})\;d\mathbf{x}\;d\mathbf{x'}.
	\]
	The asymptotic for $S_3$ is then established by inserting our asymptotic for $K$, extending the integrals back to all of $\R^k$ (at the cost of another $o(1)$ term) and the observation that
	\begin{equation}\label{alphaidentity}
	\int_{\R^k} \int_{\R^k} \left(\prod_{j = 1}^k \frac{(1 + ix_j)(1 + ix_j')}{2 + ix_j + ix_j'}\right) \tilde{F}(\mathbf{x})\tilde{F}(\mathbf{x}')\;d\mathbf{x}\;d\mathbf{x'} = \alpha_F.
	\end{equation}
Note that \eqref{alphaidentity} can be verified by differentiating \eqref{ffourier} once in each variable and applying Fubini.

The asymptotic for $S_4$ is established in a very similar way.  Defining $E : [0,\infty)^{k - 1} \rightarrow \R$ by $E(\mathbf{t}) = F(0, t_2, \ldots, t_k)$, we can proceed as before to obtain
\[
	S_4 = \int_{\R^{k - 1}}\int_{\R^{k - 1}} \left(\prod_{j = 2}^k \sum_{d, d'} \frac{\mu(d)\mu(d')}{\varphi([d,d']) |d|^{\frac{1 + ix_j}{r \log q}}|d'|^{\frac{1 + ix_j'}{r \log q}}} \right) \tilde{E}(\mathbf{x})\tilde{E}(\mathbf{x}')\;d\mathbf{x}\;d\mathbf{x'}.
\]
Expanding the inner sum as an Euler product yields
\[
	\prod_{\substack{p \in \mathcal{P} \\ p \nmid W}} \left(1 + \frac{K_p(x_j, x_j')}{\abs{p} - 1}\right),
\]
where $K_p$ is as in \autoref{eulerprod}.  We can replace each $\abs{p} - 1$ with $\abs{p}$ at the cost of a negligible error for
\[
	S_4 = (1 + o(1)) \int_{\R^{k - 1}}\int_{\R^{k - 1}} \left(\prod_{j = 2}^k K(x_j, x_j')\right) \tilde{E}(\mathbf{x})\tilde{E}(\mathbf{x'})\;d\mathbf{x}\;d\mathbf{x'}.
\]
Then the asymptotic for $S_4$ follows in the same way as that for $S_3$.
\end{proof}

With \autoref{eulerprod} and \autoref{asymptotic} established, we are now in a position to prove \autoref{npsums}.

\begin{proof}[Proof of \autoref{npsums}]
First expand the square in $S_1$ and switch the order of summation to obtain
\[
	S_1 = \sum_{d_1, \ldots, d_k, d_1', \ldots, d_k'} \lambda_{d_1, \ldots, d_k} \lambda_{d_1', \ldots, d_k'} \sum_{\substack{f \in \A_n \\ f \equiv b \mod{W} \\ [d_j,d_j'] \mid f + h_j \forall j}} 1.
\]
Recall that we are only considering $d_j, d_j'$ with $[d_1, d_1'], \ldots, [d_k, d_k'], W$ pairwise coprime.  Since $F$ is supported on the $k$-simplex, we can further restrict our attention to $d_j, d_j'$ with
\[
	\sum_{j = 1}^k \deg([d_j,d_j']) \leq 2r < n.
\]
As we've arranged for $\deg(W) \ll \log n$, we can take $w$ sufficiently large with respect to $\eta$ and apply the Chinese remainder theorem for the count
\[
	\sum_{\substack{f \in \A_n \\ f \equiv b \mod{W} \\ [d_j,d_j'] \mid f + h_j}} 1 = \frac{\abs{\A_n}}{\abs{W\prod_{j = 1}^k [d_j,d_j']}}.
\]
The asymptotic for $S_1$ then follows immediately from \autoref{asymptotic}.  For $S_2$, we begin similarly with
\[
	S_2 = \sum_{d_1, \ldots, d_k, d_1', \ldots, d_k'} \lambda_{d_1, \ldots, d_k} \lambda_{d_1', \ldots, d_k'} \sum_{\substack{f \in \A_n \\ f \equiv b \mod{W} \\ [d_j,d_j'] \mid f + h_j \forall j}} \theta(f + h)
\]
We take $h = h_1$ for simplicity; the other cases follow from the symmetry of $F$.  From the support of $F$, we certainly only need to consider $d_1, d_1'$ with $\deg([d_1, d_1']) \leq 2r < n$.  In particular, when $f + h_1$ is prime, the only nonzero summands are when $d_1 = d_1' = 1$.  Then we're left with estimating
\[
	S_2 = \sum_{d_2, \ldots, d_k, d_2', \ldots, d_k'} \lambda_{1, d_2, \ldots, d_k} \lambda_{1, d_2, \ldots, d_k} \sum_{\substack{f \in \A_n \\ f \equiv b \mod{W} \\ [d_j,d_j'] \mid f + h_j \forall j}} \theta(f + h_1),
\]
Again using our coprimality conditions and the Chinese remainder theorem, the inner sum is restricted to a single congruence class modulo $W \prod_{j = 2}^k [d_j,d_j']$.  By the strong analogue of Dirichlet's theorem, we have that the number of primes in such a congruence class is
\[
	\frac{1}{\varphi(W) \prod_{j = 2}^k \varphi([d_j,d_j'])} \cdot \frac{q^n}{n} + O\left(\frac{q^{n/2}}{n}\right).
\]
Note here that it is crucial that we use the error bounds provided by the Riemann hypothesis for curves. Then we have the estimate
\[
S_2 =  \sum_{d_2, \ldots, d_k, d_2', \ldots, d_k'} \lambda_{1, d_2, \ldots, d_k} \lambda_{1, d_2', \ldots, d_k'} \left(\frac{q^n}{\varphi(W) \prod_{j = 2}^k \varphi([d_j,d_j'])} + O\left(q^{n/2}\right)\right)
\]
From \autoref{asymptotic}, the main term for $S_2$ is as we claim.  For the error term, it is enough to note that
\[
	 \sum_{d_2, \ldots, d_k, d_2', \ldots, d_k'} \lambda_{1,d_2, \ldots, d_k} \lambda_{1, d_2', \ldots, d_k'} \ll q^r.
\]
Indeed, the support of $F$ allows us to only consider the $O(q^r)$ choices of $d_j, d_j'$ for which we have $\sum_{j = 2}^k \deg(d_j), \sum_{j = 2}^k \deg(d_j') \leq r$. This establishes the asymptotic for $S_2$ as $r < n/2$.
\end{proof}

Finally, we prove \autoref{concentration} using a similar (but more delicate) argument.

\begin{proof}[Proof of \autoref{concentration}]
We begin by fixing some prime $g \in \F_q[t]$ and analyzing the sum
\[
	S_g = \sum_{\substack{f \in \A_n \\ f \equiv b \mod{W} \\ g \mid f + h_1}} \left(\sum_{\substack{d_1, \ldots, d_k \\ d_j \mid f + h_j}} \lambda_{d_1, \ldots, d_k} \right)^2
\]
From our choice of $b$, this is zero whenever $g \mid W$, so we will only consider $g \nmid W$. Arguing as in the proof of \autoref{npsums}, we expand the square in $S_g$ and obtain
\[
	S_g = \sum_{d_1, \ldots, d_k, d_1', \ldots, d_k'} \lambda_{d_1, \ldots, d_k}\lambda_{d_1', \ldots, d_k'} \sum_{\substack{f \in \A_n \\ f \equiv b \mod{W} \\ [d_j,d_j'] \mid f + h_j \forall j \\ g \mid f + h_1}} 1.
\]
For the remainder of the argument, we will let the variable $D$ denote a squarefree, monic polynomial in $\F_q[t]$ with $g \nmid D$.  Separating the terms where $g \mid [d_1, d_1']$ and $g \nmid [d_1, d_1']$, we have
\begin{align*}
	S_g = &\sum_{D} \sum_{\substack{d_1, \ldots, d_k, d_1', \ldots, d_k' \\ [d_1,d_1'] = D}} \lambda_{d_1, \ldots, d_k}\lambda_{d_1', \ldots, d_k'} \sum_{\substack{f \in \A_n \\ f \equiv b \mod{W} \\ [d_j,d_j'] \mid f + h_j \forall j \\ g \mid f + h_1}} 1 + \sum_{\substack{d_1, \ldots, d_k, d_1', \ldots, d_k'\\ [d_1,d_1'] = gD}} \lambda_{d_1, \ldots, d_k}\lambda_{d_1', \ldots, d_k'} \sum_{\substack{f \in \A_n \\ f \equiv b \mod{W} \\ [d_j,d_j'] \mid f + h_j \forall j}} 1
\end{align*}
Note that since $g \mid f + h_1$, our choice of $b$ ensures that $g \nmid f + h_j$ for $j \neq 1$. That is, the condition $g \mid f + h_1$ ensures that $g$ and $[d_j, d_j']$ are coprime for $j \neq 1$. Applying our coprimality conditions and the Chinese remainder theorem, 
\[
S_g = \frac{\abs{\A_n}}{\abs{W}}\sum_{D} \frac{1}{\abs{g}}\sum_{\substack{d_1, \ldots, d_k, d_1', \ldots, d_k'  \\ [d_1,d_1'] = D}} \frac{\lambda_{d_1, \ldots, d_k}\lambda_{d_1', \ldots, d_k'}}{\prod_{j = 1}^k \abs{[d_j,d_j']}}
	+ \sum_{\substack{d_1, \ldots, d_k, d_1', \ldots, d_k' \\ [d_1,d_1'] = gD}} \frac{\lambda_{d_1, \ldots, d_k}\lambda_{d_1', \ldots, d_k'}}{\prod_{j = 1}^k \abs{[d_j,d_j']}}
\]
Using a Fourier expansion as in the proof of \autoref{asymptotic}, note that we have
\[
\sum_{\substack{d_1, \ldots, d_k, d_1', \ldots, d_k' \\ [d_1,d_1'] = D}} \frac{\lambda_{d_1, \ldots, d_k}\lambda_{d_1', \ldots, d_k'}}{\prod_{j = 1}^k \abs{[d_j,d_j']}} = \int_{\R^k}\int_{\R^k} \frac{K_D(x_1, x_1')}{|D|}\prod_{j = 2}^k K(x_j, x_j') \tilde{F}(\mathbf{x})\tilde{F}(\mathbf{x}')\;d\mathbf{x}\;d\mathbf{x}',
\]
where $K_D$ and $K$ are as in \autoref{eulerprod}.  Since we're only considering $D$ with $g \nmid D$, we have $K_{gD} = K_g K_D$ and the Euler product expansion 
\[
	\sum_{D} \frac{K_D(x_1,x_1')}{|D|} = K(x_1,x_1')\left(1 + \frac{K_g(x_1,x_1')}{\abs{g}}\right)^{-1}.
\]
Then by moving the sum in $D$ in our previous expression for $S_g$ inside the integrals from the Fourier expansions, we arrive at
\[
S_g = \frac{\abs{\A_n}}{\abs{gW}} \int_{\R^k} \int_{\R^k} \frac{1 + K_g(x_1,x_1')}{1 + K_g(x_1, x_1')/\abs{g}} \prod_{j = 1}^k K(x_j,x_j') \tilde{F}(\mathbf{x})\tilde{F}(\mathbf{x}')\;d\mathbf{x}\;d\mathbf{x'}
\]
Arguing as in \autoref{asymptotic}, we arrive at
\[
	S_g = \left(1 + o(1)\right)\frac{I_g}{\abs{g}}\frac{\abs{W}^{k - 1}}{\varphi(W)^k} \frac{\abs{\A_n}}{r^k},
\]
where
\[
	I_g = \int_{|\mathbf{x}| \leq \sqrt{r}} \int_{|\mathbf{x'}| \leq \sqrt{r}} \frac{1 + K_g(x_1,x_1')}{1 + K_g(x_1, x_1')/\abs{g}} \prod_{j = 1}^k \frac{(1 + ix_j)(1 + ix_j')}{2 + ix_j + ix_j'} \tilde{F}(\mathbf{x})\tilde{F}(\mathbf{x'})\;d\mathbf{x}\;d\mathbf{x'}.
\]
We need to estimate
\[
	\frac{1 + K_g(x_1,x_1')}{1 + K_g(x_1, x_1')/\abs{g}} = \frac{\left(1 - \abs{g}^{-\frac{1 + ix_1}{r \log q}}\right)\left(1 - \abs{g}^{-\frac{1 + ix_1'}{r \log q}}\right)}{1 - \left(\abs{g}^{-1 - \frac{1 + ix_1}{r \log q}} + \abs{g}^{-1 - \frac{1 + ix_1}{r \log q}} - \abs{g}^{-1 - \frac{2 + ix_1 + ix_1'}{r \log q}}\right)}
\]
From
\[
	\frac{1}{1 - \left(\abs{g}^{-1 - \frac{1 + ix_1}{r \log q}} + \abs{g}^{-1 - \frac{1 + ix_1'}{r \log q}} - \abs{g}^{-1 - \frac{2 + ix_1 + ix_1'}{r \log q}}\right)} = 1 + O\left(\frac{1}{\abs{g}}\right) = 1 + o(1),
\]
and
\[
	\abs{g}^{-\frac{1 + ix_1}{r \log q}} = 1 - \frac{\deg(g)}{r} (1 + ix_1) + O\left( \frac{\deg(g)^2}{r^2} (1 + ix_1)^2\right),
\]
we have (using $\abs{x_1}, \abs{x_1'} \leq \sqrt{r}$)
\[
	\frac{1 + K_g(x_1,x_1')}{1 + K_g(x_1, x_1')/\abs{g}} = (1 + o(1))\frac{\deg(g)^2}{r^2}(1 + ix)(1 + ix')
\]
Our asymptotic for $S_g$ becomes
\[
	S_g = (\gamma_F + o(1))\frac{\deg(g)^2}{\abs{g}r^2} \frac{\abs{W}^{k - 1}}{\varphi(W)^k} \frac{\abs{\A_n}}{r^k}
\]
where 
\[
	\gamma_F = \int_{[0,\infty)^k} \frac{\partial F}{(\partial t_1)^2 \partial t_2 \cdots \partial t_k}(\mathbf{t})^2\;d\mathbf{t}.
\]
Note that one can verify the identity
\[
	\gamma_F = \int_{\R^k} \int_{\R^k} (1 + ix_1)(1 + ix_1') \prod_{j = 1}^k \frac{(1 + ix_j)(1 + ix_j')}{2 + ix_j + ix_j'} \tilde{F}(\mathbf{x})\tilde{F}(\mathbf{x'})\;d\mathbf{x}\; d\mathbf{x}'.
\]	
as before by differentiating \eqref{ffourier} twice in $t_1$ and once in each other variable.  It remains to sum $S_g$ for primes $g$ with $\deg(g) \leq \epsilon n$.  By the symmetry of $F$, we have the estimate
\[
	\sum_{\substack{f \in \A_n \\ f \equiv b \mod{W} \\ P^-\left(\prod_{j = 1}^k (f + h_j)\right) \leq \epsilon n}} \left(\sum_{\substack{d_1, \ldots, d_k \\ d_j \mid f + h_j}} \lambda_{d_1, \ldots, d_k} \right)^2 \ll  \frac{\abs{W}^{k - 1}}{\varphi(W)^k} \frac{(\log q)^k}{r^k} \abs{\A_n} \sum_{ \substack{g \in \mathcal{P} \\ \deg(g) \leq \epsilon n} } \frac{\deg(g)^2}{\abs{g}r^2}.
\]
For $\epsilon < \eta$, we can apply the prime number theorem to see
\[
	\sum_{\substack{g \in \mathcal{P} \\ \deg(g) \leq \epsilon n}} \frac{\deg(g)^2}{\abs{g}r^2} \ll \frac{\epsilon n}{r} \ll \epsilon
\]
completing the proof.

\end{proof}

\section{The Transference Argument}\label{transsection}

For this section, we consider the quotient of $\F_q[t]$ by a fixed monic polynomial of degree $n$, so that we can make the identification between $\{f \in \F_q[t] : \deg(f) < n\}$ and $\F_{q^n}$.  Addition in these spaces agree under this identification, but we will need to be careful when we want multiplication in $\F_{q^n}$ to be valid in $\F_q[t]$.

A major insight of Green and Tao was that, loosely speaking, one should be able to transfer Szemer\'edi's theorem on arithmetic progressions from dense subsets of integers to dense subsets of any sufficiently ``pseudorandom'' subset of integers. This type of transference was established in $\F_q[t]$ by L\^e \cite{le}, but we will apply the more recent ``relative Szemer\'edi theorem'' established by Conlon, Fox, and Zhao \cite{cfz}. We use their notion of pseudorandom to make the following definition.

\begin{defn} We say that a function $\nu : \F_{q^n} \rightarrow [0,\infty)$ is an $\ell$-pseudorandom measure if for any choice of exponents $c_{j,\omega} \in \{0,1\}$, one has
\[
	\E \left[ \prod_{j = 1}^\ell \prod_{\omega \in \{0,1\}^{[1,\ell] \setminus \{j\}}} \nu\left(\sum_{i = 1}^\ell (i - j)x_{i, \omega_i}\right)^{c_{j,\omega}} \right] = 1 + o(1)
\]
where here the expectation is taken over all $x_{1, 0}, x_{1, 1}, x_{2,0}, x_{2,1}, \ldots, x_{\ell,0}, x_{\ell,1} \in \F_{q^n}$.

\end{defn}

In our setting, the relative Szemer\'edi theorem \cite[Theorem 3.1]{cfz} reduces to \autoref{ffgreentao} below.  This is essentially \cite[Theorem 4]{le}, in which $\nu$ was required to satisfy a stronger notion of pseudorandom.  The main breakthrough of \cite{cfz} was to show that one can establish theorems of this type under their weaker notion of pseudorandom.

\begin{thm}\label{ffgreentao}
For every $\ell \in \N$ and $\delta > 0$, there exists $c > 0$ such that if $\nu : \F_{q^n} \rightarrow \R$ is an $\ell$-pseudorandom measure, then for $n$ sufficiently large, we have for every $\phi : \F_{q^n} \rightarrow \R$ with $0 \leq \phi \leq \nu$ and $\E_{f \in \F_{q^n}} \phi(f) \geq \delta$ the count
\[
	\E_{f,g \in \F_{q^n}} \prod_{\substack{h \in \F_q[t] \\ \deg(h) < \ell}} \phi(f + gh) \geq c.
\]
\end{thm}

\begin{rem}
	Here by $n$ sufficiently large, we mean with respect to $\delta$, $\ell$, and the $o(1)$ term in the definition of $\ell$-pseudorandom.  
\end{rem}

Remember that our goal is to show that $\mathcal{A}$ from \autoref{positivedensity} contains large $\ell$-configurations.  It is tempting at this stage to try to apply \autoref{ffgreentao} with $\phi = 1_\mathcal{A}$.  This approach is doomed to fail, since for any $\mathcal{H}$ and $\epsilon > 0$, the set $\mathcal{P}_\epsilon(\mathcal{H})$ has density zero in $\F_q[t]$.  To remedy this, we need to show that we can build a pseudorandom measure $\nu$ that controls an appropriately weighted version of $1_\mathcal{A}$.  This is contained in the following theorem.

\begin{thm}\label{findmeasure}
	Let $\mathcal{A} \subseteq \{f \in \mathcal{P}_\epsilon(\mathcal{H}) \cap \A_{n} : f \equiv b \mod{W}\}$ and $a \in (0,1)$ with
	\[
		\abs{\mathcal{A}} \geq a \frac{\abs{W}^{k - 1}}{\varphi(W)^k} \frac{\abs{\A_n}}{n^k} .
	\]
	Then there exists some $\delta > 0$, an $\ell$-pseudorandom measure $\nu : \F_{q^n} \rightarrow [0,\infty)$, and a function $\phi : \F_{q^n} \rightarrow [0,\infty)$ such that:
	\begin{enumerate}
	\item[(i)] $\phi(f) = 0$ unless $Wf + b \in \mathcal{A}$ when we consider $f \in \F_q[t]$
	\item[(ii)] $\phi(f) \leq \nu(f)$ for all $f \in \F_{q^n}$
	\item[(iii)] $\E_{f \in \F_{q^n}} \phi(f) \geq \delta$
	\item[(iv)] $\norm{\phi}_\infty \leq n^k$.
\end{enumerate}
\end{thm}

\begin{rem}
	If $\epsilon = 1$ and $\mathcal{H} = \{0\}$, then $\mathcal{P}_\epsilon(\mathcal{H}) = \mathcal{P}$ and \autoref{findmeasure} reduces to \cite[Theorem 5]{le}.  
\end{rem}

\begin{proof}
We take $G : [0,\infty) \rightarrow \R$ smooth and supported on $[0,1]$ such that $G(0) = 1$ and $\int_0^\infty G'(t)^2 \;dt = 1$.  For $r \in (0,n)$, define the Goldston-Y\i ld\i r\i m divisor sum
\[
	\Lambda_r(f) = \sum_{\substack{d \in \F_q[t] \text{ monic} \\ d \mid f}} \mu(d)G\left(\frac{\deg(d)}{r}\right).
\]

In \cite{le}, L\^e establishes that the measures $\nu_j(f) = \frac{r \varphi(W)}{|W|} \Lambda_r(Wf + b + h_j)^2$
are $\ell$-pseudorandom provided $(b + h_j, W) = 1$ (which we've arranged for) and $r = \rho n$ with $\rho \in (0,1)$ sufficiently small with respect to $\ell$.  We further insist that $\rho < \epsilon$.  We take the measure $\nu(f) = \prod_{j = 1}^k \nu_j(f)$.  The proof that $\nu$ is $\ell$-pseudorandom can be run completely analogously to L\^e's proof that $\nu_j$ satisfies his (stronger) linear forms condition. The only missing ingredient is an analogue of \cite[Proposition 11]{le}.  To establish such an analogue, the only necessary observation is that, for any $p \in \mathcal{P}$ with $p \nmid W$, we cannot have $p \mid Wf + b + h_j$ and $p \mid Wf + b + h_{j'}$ simultaneously for $j \neq j'$.  Indeed, if this were the case, we would have $p \mid h_j - h_{j'}$, which is a contradiction since $w > \deg(h_k)$.  This is essentially the same observation that Pintz made to verify the linear forms condition in his proof of \cite[Theorem 5]{pintz}.

We define $\phi : \F_{q^n} \rightarrow \R$ by
\[
	\phi(f) = \begin{cases} \frac{r^k \varphi(W)^k}{\abs{W}^k} & \text{if } Wf + b \in \mathcal{A} \\ 0 &\text{otherwise} \end{cases}.
\]
It should be immediately clear that $\phi$ satisfies (i) and (iv).  To verify (ii), it is enough to note that for $f \in \F_q[t]$ with $Wf + b \in \mathcal{A}$, we have $\phi(f) = \nu(f)$.  Finally, note that
\[
	\E_{f \in \F_{q^n}} \phi(f) = \frac{1}{q^n} \frac{r^k \varphi(W)^k}{\abs{W}^k} \abs{\{f \in \F_q[t] : Wf + b \in \mathcal{A}\}} \geq \frac{a\rho^k}{\abs{W}},
\]
so by setting $\delta = a\rho^k/\abs{W}$, we've verified (iii).
\end{proof}

\section{Proof of Theorem 2}\label{proofmain}

With the density and transference arguments established, we are ready to prove \autoref{headline}. By \autoref{positivedensity}, we can fix $w \in \N$ sufficiently large and $\epsilon$ sufficiently small to obtain some $a > 0$ such that, for $n$ sufficiently large with respect to $w$, the set
\[
		\mathcal{A} = \left\{f \in \mathcal{P}_\epsilon(\mathcal{H}) \cap \A_n : f \equiv b \mod{W} \text{ and at least } m + 1 \text{ of the } f + h_j \text{ are prime}\right\}
\]
satisfies
\[
		\abs{\mathcal{A}} \geq a \frac{\abs{W}^{k - 1}}{\varphi(W)^k} \frac{\abs{\A_n}}{n^k} .
\]
Then from \autoref{findmeasure}, there exists some $\delta > 0$, an $\ell$-pseudorandom measure $\nu : \F_{q^n} \rightarrow [0,\infty)$, and a function $\phi : \F_{q^n} \rightarrow [0,\infty)$ such that (i) - (iv) of \autoref{findmeasure} hold. If we applied \autoref{ffgreentao} directly, we could use this to locate $\ell$-configurations with multiplication in $\F_{q^n}$, but these are not necessarily $\ell$-configurations in $\F_q[t]$. To remedy this, fix $\ell \in \N$ and partition $\F_{q^n}$ into $q^\ell$ disjoint sets $C_i$, each containing $q^{n - \ell}$ elements with $\deg(f - g) < n - \ell$ for any $f,g \in C_i$. Then $\E_{f \in C_i} \phi(f) \geq \delta$ for some $i$, so we fix this $C_i$ and set $\psi = \phi|_{C_i}$. Note $\psi$ still satisfies (i), (ii), and (iv), and almost satisfies (iii) in the sense that $\E_{f \in \F_{q^n}} \psi(f) \geq \delta / q^\ell$. We now apply \autoref{ffgreentao} to find a constant $c > 0$ such that
\[
	\E_{f, g \in \F_{q^n}} \prod_{\substack{h \in \F_q[t] \\ \deg(h) < \ell}} \psi(f + gh) \geq c.
\]
From (iv), the contribution of the terms where $g = 0$ is negligible, in the sense that we have the same average with nonzero $g$ at least $c/2$ by taking $n$ sufficiently large. Note that if $f \in C_i$ and $f + g \in C_i$, then $\deg(g) < n - \ell$, so we have that $\deg(f + gh) < n$ for all $h$ with $\deg(h) < \ell$.  That is, considering $\psi$ as a function on $\{f \in \F_q[t] : \deg(f) < n\}$, we have shown that
\[
	\sum_{\substack{f,g \in \F_q[t] \\ \deg(f) < n \\ \deg(g) < n,\;g \neq 0}} \prod_{\substack{h \in \F_q[t] \\ \deg(h) < \ell}} \psi(f + gh) > 0.
\]
It follows that we can find $f,g \in \F_q[t]$ with $g \neq 0$ and 
\[
	\mathcal{C}_\ell(f,g) \subseteq \{h \in \F_q[t] : Wh + b \in \mathcal{A}\}.
\]
Then each of the $\ell$-configurations $h_j + \mathcal{C}_\ell(Wf + b, Wg)$ contain only monic elements of degree $n$, and at least $m + 1$ of them contain only irreducibles. As we can produce similar collections of $\ell$-configurations for all sufficiently large $n$, we've established \autoref{headline}.

\section{Concluding Remarks}

A careful reading of the proof of \autoref{headline} indicates that we have proved something slightly stronger.  Indeed, for every sufficiently large degree $n$, we were able to produce the required collections of $\ell$-configurations whose elements were all of degree $n$.  It follows that by adapting the argument of \cite[Theorem 1.4(i)]{chlopt}, one could strengthen \autoref{headline-quant} to be true for \emph{any} monomial of degree $d$.

\emph{Acknowledgments.} The author would like to thank Neil Lyall and \`{A}kos Magyar for many helpful conversations.  This work also benefited from discussions with Paul Pollack and Lee Troupe.

\bibliography{SGBCOPP}{}
\bibliographystyle{abbrv}

\end{document}